\documentclass[11pt,reqno]{amsart}
\usepackage{amsmath,amssymb,amsthm,amsfonts,enumerate,hyperref}
\theoremstyle{plain}
\newtheorem{theorem}{Theorem}[section]
\newtheorem{proposition}[theorem]{Proposition}

\theoremstyle{definition}

\theoremstyle{remark}

\numberwithin{equation}{section}

\newcommand{\Ric}{\mathrm{Ric}}

\newcommand{\Hess}{\mathrm{Hess}}

\begin{document}
\title{Perelman's entropy for some families of canonical metrics}
\author{Stuart J. Hall}
\address{Department of Applied Computing, University of Buckingham, Hunter St., Buckingham, MK18 1G, U.K.} 
\email{stuart.hall@buckingham.ac.uk}
\maketitle
\begin{abstract}
We numerically calculate Perelman's entropy for a variety of canonical metrics on $\mathbb{CP}^{1}$-bundles over products of Fano K\"ahler-Einstein manifolds. The metrics investigated are Einstein metrics, K\"ahler-Ricci solitons and quasi-Einstein metrics. The calculation of the entropy allows a rough picture of how the Ricci flow behaves on each of the manifolds in question.
\end{abstract}

\section{Introduction}
\subsection{Background} 
The purpose of this article is to investigate the behaviour of the Ricci flow through the lens of Perelman's entropy, which he introduced in \cite{Per}. On manifolds that admit a variety of metrics which are fixed points of the flow (such metrics are called \emph{Ricci solitons}) it is a very natural question to ask whether one can produce Ricci flows that flow between these geometries. Perelman introduced a functional on the space of Riemannian metrics (the $\nu$-functional) which is monotone increasing under Ricci flow and stationary only at a Ricci soliton.  This allows one to `order' the various geometries on a given manifold and in particular provides a crude notion of how the geometry of the manifold should evolve under a Ricci flow.\\ 
\\
The manifolds we will investigate are known as equidistant hypersurface families. These are compact manifolds $\overline{M}^{n+1}$ that admit a smooth function $t$ such that $t(\overline{M})=I$ for a closed interval $I$ and $t^{-1}(I)$ is diffeomorphic to $I\times P$ for a smooth $n$-dimensional manifold $P$. In the case considered in this article, the manifolds will be Fano varieties that contain principal circle bundles over the product of Fano K\"ahler-Einstein metrics as the hypersurfaces $P$. One notable member of this family is the one-point blow-up of the complex projective plane, $\mathbb{CP}^{2}\sharp-\mathbb{CP}^{2}$ (also known as the first Hirzebruch surface $F_{1}$).\\
\\
Many such manifolds admit a variety of canonical metrics; we list these examples below (see section \ref{ansatz} for more detailed descriptions):
\begin{enumerate}
\item A non-K\"ahler, Einstein metric with fibre-wise $\mathbb{Z}_{2}$-symmetry, due to Page \cite{Pag}, B\'erard-Bergery \cite{BB} and Wang and Wang \cite{WW}.
\item A non-K\"ahler, Einstein metric with no $\mathbb{Z}_{2}$-symmetry, due to Wang and Wang \cite{WW}.
\item A shrinking gradient K\"ahler-Ricci soliton, due to Koiso \cite{Koi}, Cao \cite{Cao}, Chave and Valent \cite{CV} and Dancer and Wang \cite{DW}.
\end{enumerate}
Recently L\"u, Page and Pope \cite{LPP} and the author \cite{Hall13} also produced examples of quasi-Einstein metrics on these manifolds.  These are not themselves fixed points of the flow but form the base of warped product Einstein metrics.  Hence it is also interesting to consider the evolution of the geometry of the products $M\times F$ under the Ricci flow.\\
\\
Motivation for this work comes from calculations done on the entropy of 4-manifold geometries by Cao, Hamilton and Ilmanen \cite{CHI} and the author \cite{Hall11}. The behaviour of the K\"ahler-Ricci flow on projective bundles has  been investigated \cite{Song} where it is shown that the flow shrinks the $\mathbb{CP}^{1}$ fibres. The behaviour of the Ricci flow is far more complicated and this work allows a very rough picture to develop of the evolution of the geometry on the manifolds investigated.  The author is not aware of any work that explicitly constructs a flow between any of the canonical metrics on these manifolds. This work suggests sensible metrics to try and connect with Ricci flows.\\
\\
As we will discuss in section \ref{ansatz}, the ansatz for each of the canonical metrics is essentially determined by a single numerical parameter. A secondary contribution of this paper is to collect these ansatz and parameters together in one place so that one could easily use these metrics to test other properties and conjectures about Einstein metrics and Ricci solitons.
\subsection{Results}
After describing the ansatz for each family of metrics and discussing how to calculate the entropy, we give a detailed study of four different manifolds.  The first is the non-trivial $\mathbb{CP}^{1}$-bundle over $\mathbb{CP}^{1}$. This manifold admits the Einstein metric of type (1) (originally due to Page \cite{Pag}), a K\"ahler-Ricci soliton of type (3) (originally due to Koiso  \cite{Koi} and Cao \cite{Cao}) and quasi-Einstein metrics (due to  L\"u, Page and Pope \cite{LPP}). The next two manifolds are non-trivial $\mathbb{CP}^{1}$-bundle over $\mathbb{CP}^{2}$. They both admit metrics of type (1), (3) and quasi-Einstein metrics. The final example is a non-trivial $\mathbb{CP}^{1}$-bundle over $\mathbb{CP}^{1}\times \mathbb{CP}^{2}$. This manifold admits metrics of types (1), (2) and (3) as well as two non-isometric families of quasi-Einstein metrics.\\
\\
We refer the reader to section \ref{Results} for detailed numerical results but we can summarize some of the findings here:
\begin{itemize}
\item The Einstein metrics with $\mathbb{Z}_{2}$ fibre-wise symmetry have the lowest entropy in each case.
\item We exhibit warped product Einstein manifolds where the warped Einstein metrics have both lower and higher entropy than ordinary product solitons.
\item We recover through the entropy, the convergence of the quasi-Einstein metrics to the K\"ahler-Ricci soliton.
\end{itemize} 
The second finding is slightly unexpected. Roughly speaking, the Ricci flow is expected to `break apart' manifolds and flowing from a product metric towards a non-product metric is contrary to this behaviour. One might wonder whether warped product metrics could ever be stable (attracting) as fixed points of the flow (stable metrics tend to have higher entropies).\\
\\
The remainder of the paper is structured as follows. In section \ref{RFRSPE} we give some background on the Ricci flow and Perelman's entropy. In section \ref{ansatz} we give the ansatz for each type of metric we consider in this paper. In section \ref{Calcs} we relate the ansatz to the calculation of Perelman's entropy. Finally, in section \ref{Results} we give detailed numerical results on the four manifolds described above.
\section{Ricci flow, Ricci solitons and Perelman's entropy}\label{RFRSPE}
\subsection{Ricci flow and Ricci solitons}
A \emph{Ricci flow} is a one parameter family of complete Riemannian manifolds $(M,g(t))$ that satisfy
\begin{equation}\label{RFeq}
\frac{\partial g}{\partial t}=-2\Ric(g).
\end{equation}
A \emph{Ricci soliton} is a complete Riemannian manifold $(M,g)$ and a vector field $X$ that satisfy
\begin{equation}\label{RSeq}
\Ric(g)+\frac{1}{2}\mathcal{L}_{X}g=\frac{1}{2\tau}g,
\end{equation}
for some constant $\tau$.  In the case that the vector field $X=\nabla f$ for a smooth function $f$ (defiined up to a constant), we obtain a \emph{gradient Ricci soliton} that satisfies
\begin{equation}\label{GRSeq}
\Ric(g) + \Hess (f)= \frac{1}{2\tau}g.
\end{equation}
In this article we shall be interested in solitons with $\tau>0$ which are referred to as \emph{shrinking Ricci solitons} (or shrinkers). We note that if the function $f$ is constant in equation (\ref{GRSeq}) then we recover the definition of an Einstein metric with positive scalar curvature, Einstein  metrics are therefore referred to as trivial Ricci solitons.  
A Ricci soliton evolves under the Ricci flow (\ref{RFeq}) by simultaneous scaling and the one parameter family of diffeomorphisms generated by the vector field $X$. Hence one does not consider the underlying geometry of the manifold to be changing and they are regarded as fixed points of the Ricci flow.
\subsection{Quasi-Einstein metrics and warped products}
A \emph{quasi-Einstein metric} is a complete Riemannian manifold $(M,g)$ that satisfies
\begin{equation}\label{QEMeq}
\Ric(g)+\Hess (f)-\frac{1}{m}df\otimes df=\frac{1}{2\tau}g,
\end{equation}
for a smooth function $f$ and constants $m>0$ and $\tau$. Again if the function $f$ is constant we recover the notion of an Einstein metric. Also, formally letting $m \rightarrow \infty$ we recover the gradient Ricci soliton equation (\ref{GRSeq}).\\
\\
Given a pair of Riemannian manifolds $(B,g)$ and $(F,h)$ and a smooth function ${f:B\rightarrow \mathbb{R}}$, the \emph{warped product metric} is a Riemannian metric $g\oplus e^{2f}h$ on the product manifold $B\times F$ given by
$$ g\oplus e^{2f}h = \pi^{\ast}_{B}g\oplus e^{2f}\pi_{F}^{\ast}h,$$
where $\pi_{B}$ and $\pi_{F}$ are the obvious projections from the product $B\times F$ onto each factor. We will use the shorthand $B\times_{f} F$ to denote a warped product as described above. Warped products are related to quasi-Einstein metrics in the following manner:
\begin{theorem}[Kim-Kim \cite{KK}]
A Riemannian manifold $(M,g)$ and a positive integer $m$ satisfy the quasi-Einstein condition (\ref{QEMeq}) if and only if  ${M\times_{f} F^{m}}$ is an Einstein metric, where $(F^{m}, h)$ is an Einstein metric with Einstein constant $\mu$ satisfying
\begin{equation}\label{mueq}
\mu e^{2f/m} = \frac{1}{2\tau}-\frac{1}{m}(\Delta f-|\nabla f|^{2}).
\end{equation}
\end{theorem}

\subsection{Perelman's $\nu$-functional}
Perelman \cite{Per} defined the following function called the $\mathcal{W}$-entropy.  Let $(M^{n},g)$ be a compact Riemannian manifold, $f$ be a smooth function and $\tau>0$ a constant. Then the $\mathcal{W}$-entropy is given by
\begin{equation}\label{Wfunc}
\mathcal{W}(g,f,\tau) = (4\pi\tau)^{-n/2}\int_{M}[\tau(R+|\nabla f|^{2})+f-n]e^{-f}dV_{g}
\end{equation}
where $R$ is the scalar curvature of the metric $g$. Using this functional one can define the $\nu$-functional as the infimum over compatible pairs $(f, \tau)$
\begin{equation}\label{nufunc}
\nu(g) = \inf_{(f,\tau)}\{\mathcal{W}(g,f,\tau) \ | \ (4\pi\tau)^{-n/2}\int_{M}e^{-f}dV_{g}=1\}.
\end{equation}
It can be shown, following the work of Rothaus \cite{Rot}, that for any metric $g$ this infimum is achieved by a pair $(f,\tau)$ satisfying the equations
$$\tau(-2\Delta f +|\nabla f |^{2}-R)-f+n+\nu(g)=0 \textrm{ and } (4\pi\tau)^{-n/2}\int_{M}fe^{-f}dV_{g}=\frac{n}{2}+\nu(g).$$
The main utility of this functional is the following:
\begin{theorem}[Perelman, \cite{Per}]
The $\nu$-functional is monotone increasing under the Ricci flow and stationary if and only if $g$ is a gradient Ricci soliton satisfying equation (\ref{GRSeq}).
\end{theorem}
In particular, if $g$ is a gradient Ricci soliton, the function $f$ in equation (\ref{GRSeq}) is the one that achieves the infimum in the definition of the $\nu$-entropy. Hence we have the following formula for the $\nu(g)$-entropy of a Ricci soliton.
\begin{equation}\label{solnueq}
\nu(g)=(4\pi\tau)^{-n/2}\int_{M}fe^{-f}dV_{g}-\frac{n}{2}.
\end{equation}
In the case when $g$ is Einstein (i.e. $f$ is constant), the $\nu$-entropy is given by the simple formula
$$\nu(g)=\log\left(\frac{Vol(M,g)}{(4\pi\tau e)^{n/2}}\right).$$
If $(M_{1},g_{1})$ and $(M_{2},g_{2})$ are both Ricci solitons with the same constant $1/2\tau$ then the product metric ${g_{1}\oplus g_{2}}$ is also a soliton with constant $1/2\tau$ on $M_{1}\times M_{2}$.  The $\nu$-entropy is additive in this case so
$$ \nu(g_{1}\oplus g_{2})=\nu(g_{1})+\nu(g_{2}).$$
\section{Ansatz for the hypersurface families}\label{ansatz}
The examples of canonical metrics studied in this article are constructed on equidistant hypersurface families. Let $(V_{i}^{2n_{i}},J_{i},h_{i}), 1\leq i \leq r$ be compact Fano, K\"ahler-Einstein manifolds of real dimension $2n_{i}$ where the first Chern class $c_{1}(V_{i},J_{i})=p_{i}a_{i}$, where $p_{i}$ are integers and $a_{i}\in H^{2}(V_{i}, \mathbb{Z})$ are indivisible cohomology classes. The K\"ahler-Einstein metrics are normalised so that $\Ric (h_{i}) = p_{i}h_{i}$. For ${q=(q_{1},...,q_{r})}$ with $q_{i}\in \mathbb{Z}/\{0\}$ and let $P_{q}$ be the principal $U(1)$-bundle over $V:=V_{1}\times ...\times V_{r}$ with Euler class $\sum_{i=1}^{i=r}q_{i}\pi^{\ast}_{i}a_{i}$ and $\pi_{i}$ is the projection from $V$ onto the $i$th factor.\\ 
\\
Let $\theta$ be the principal $U(1)$ connection on $P_{q}$ with curvature $\Omega=\sum_{i=1}^{i=r}  q_{i}\pi^{\ast}_{i}\eta_{i}$ where $\eta_{i}$ is the K\"ahler form of the metric $h_{i}$. We consider metrics of the form
\begin{equation}\label{met}
\alpha^{-1}ds^{2}+\alpha(s)\theta \otimes \theta +\sum_{i=1}^{i=r}\beta_{i}(s)\pi^{\ast}_{i}h_{i},
\end{equation}
on the manifold $I\times P_{q}$ where $\alpha$ and $\beta_{i}$ are smooth functions on a bounded open interval $I$.
The manifold $M_{0}$ is compactified by collapsing a circle at each end of the interval $I$.  If we denote the closure of the interval $I$ by $\bar{I}=[0,s^{\ast}]$ then in order for metrics of the form (\ref{met}) to extend smoothly we require
$$\alpha(0)=\alpha(s^{\ast})=0, \  \alpha'(0)=-\alpha'(s^{\ast})=2$$
and the $\beta_{i}$ to be positive on $\bar{I}$. In order to construct the various different metrics we first factor out homothety by setting $\tau=1$ for the remainder of this paper.
\subsection{The Wang-Wang Einstein metrics with fibre-wise $\mathbb{Z}_{2}$-symmetry}
The metrics here are given by Theorem 1.4 in \cite{WW}.  We will be interested in the lift of these Einstein metrics to the covering $\mathbb{CP}^{1}$-bundle which yield Hermitian, non-K\"ahler, Einstein metrics with fibre-wise $\mathbb{Z}_{2}$-symmetry.
The only condition we require for existence in this case is ${0<|q_{i}|<p_{i}}$. The functions $\alpha$ and $\beta_{i}$ are given by:
$$\alpha(s) = \frac{(R-s)}{\prod_{i=1}^{i=r}\left(\beta_{i}(s)\right)^{n_{i}}}\int_{0}^{s}\left[\frac{1}{2}-\frac{R(R-4)}{2(R-x)^{2}}\right]\prod_{i=1}^{i=r}\left(\beta_{i}(x)\right)^{n_{i}}dx$$
and 
$$\beta_{i}(s) = A_{i}(s-R)^{2}-\frac{q_{i}^{2}}{4A_{i}}.$$
The constants $A_{i}$ and $R>0$ satisfy
$$\frac{R}{2}(R-4) = \frac{8A_{i}p_{i}+q_{i}^{2}}{8A_{i}^{2}}.$$
Given a particular value of $R$ one choses the negative root of the quadratic determined by the above condition to yield $A_{i}$.
The constant $R$ is determined by requiring $\alpha'(R)=0.$

\subsection{The Wang-Wang Einstein metrics}
The precise existence theorem covering this case is Theorem 1.2 in \cite{WW}. Here the assumptions are ${0<|q_{i}|<p_{i}}$ and that for some choice of ${(\epsilon_{1},...,\epsilon_{r})}$ with $\epsilon_{i}=\pm 1$ and at least one $\epsilon_{i}=1$, the integral
\begin{equation}\label{futakWW}
\int_{-1}^{1}\left[\prod_{i=1}^{i=r}\left(\frac{p_{i}}{|q_{i}|}+\epsilon_{i}x\right)^{n_{i}}\right]xdx<0.
\end{equation}
This \emph{a fortiori} implies that $r>1$.
The functions $\alpha$ and $\beta_{i}$ have the form
$$\alpha(s) = \frac{(s+\kappa_{0})}{\prod_{i=1}^{i=r}\beta_{i}^{n_{i}}(s)}\int_{0}^{s}\left(E^{\ast}-\frac{(x+\kappa_{0})^{2}}{2}\right)(x+\kappa_{0})^{-2}\prod_{i=1}^{i=r} \beta_{i}^{n_{i}}(x) dx,$$
$$\beta_{i}(s) = A_{i}(s+\kappa_{0})^{2}-\frac{q_{i}^{2}}{4A_{i}}.$$
Here the constants satisfy
$$E^{\ast} = \frac{\kappa_{0}(\kappa_{0}+4)}{2} =  \frac{8A_{i}p_{i}+q_{i}^{2}}{8A_{i}^{2}}.$$
If the $\epsilon_{i}=1$ in (\ref{futakWW})  then one takes the positive root of the quadratic to yield $A_{i}$.  If $\epsilon_{i}=-1$ then one takes the negative root. The constant $\kappa_{0}$ is determined by the requirement that $\alpha(4)=0$.
\subsection{K\"ahler-Ricci solitons}
Here the precise existence theorem is Proposition 4.25 in \cite{DW}.  The ansatz is given by:
$$\alpha(s) = \frac{e^{\kappa_{1}s}}{\prod_{i=1}^{i=r}\left(s-2-\frac{2p_{i}}{q_{i}} \right)^{n_{i}}}\int_{0}^{s}(2-x)e^{\kappa_{1}x}\prod_{i=1}^{i=r}\left(x-2-\frac{2p_{i}}{q_{i}} \right)^{n_{i}}dx,$$
$$\beta_{i}(s)=-q_{i}(s+\sigma_{i}).$$
The soliton potential function is given by
$$f(s)=\kappa_{1}(s+\kappa_{0}).$$
The constants $\sigma_{i}, \kappa_{0}$ and $\kappa_{1}$ satisfy the consistency conditions
$$2=\frac{C}{\kappa_{1}}-\kappa_{0}=-\sigma_{i}-2\frac{p_{i}}{q_{i}},$$
where $C$ is a constant that reflects the fact one can choose an arbitrary value of $f$.
The constant $\kappa_{1}$ is determined by the requirement $\alpha(4)=0$. We will only be interested in the case when $\kappa_{1} \neq 0$ as these correspond to non-trivial K\"ahler-Ricci solitons.
\subsection{Quasi-Einstein metrics}\label{QEMans}
The precise existence theorem covering this case is Theorem 3 in \cite{Hall13}. Here the conditions for existence are ${0<|q_{i}|<p_{i}}$ and that for some choice of ${(\epsilon_{1},...,\epsilon_{r})}$ with $\epsilon_{i}=\pm 1$, the integral
\begin{equation} \label{futakQEM}
\int_{-1}^{1}\left[\prod_{i=1}^{i=r}\left(\frac{p_{i}}{|q_{i}|}+\epsilon_{i}x\right)^{n_{i}}\right]xdx<0.
\end{equation}
In this case the interval $I=(0,4)$ and the functions $\alpha$ and $\beta_{i}$ have the form
$$\alpha(s) = \frac{(s+\kappa_{0})^{1-m}}{\prod_{i=1}^{i=r}\left( \beta_{i}(s)\right)^{n_{i}}}\int_{0}^{s}(x+\kappa_{0})^{m-2}\left(E^{*}-\frac{(x+\kappa_{0})^{2}}{2}\right)\prod_{i=1}^{i=r}\left( \beta_{i}(x)\right)^{n_{i}}dx$$
and
$$\beta_{i}(s) = A_{i}(s+\kappa_{0})^{2}-\frac{q_{i}^{2}}{4A_{i}},$$
where $A_{i},\kappa_{0}$ and $E^{\ast}$ are all constants.  The function $f$ in equation (\ref{QEMeq}) is given by
$$e^{-f/m} = \kappa_{1}(s+\kappa_{0}),$$
for a constant $\kappa_{1}$.  The constants satisfy the consistency conditions
$$E^{\ast} =  \frac{\mu}{\kappa_{1}^{2}}=\frac{\kappa_{0}}{2}(\kappa_{0}+4)=\frac{8A_{i}p_{i}+q_{i}^{2}}{8A_{i}^{2}}$$
where $\mu$ is the constant appearing in (\ref{mueq}). As with the Wang-Wang Einstein metrics, the sign for each $A_{i}$ is determined by the sign of $\epsilon_{i}$ appearing in the integral (\ref{futakQEM}). The constant $\kappa_{0}$ can be determined by requiring
$$\int_{0}^{4}(x+\kappa_{0})^{m-2}\left(E^{\ast}-\frac{(x+\kappa_{0})^{2}}{2}\right)\prod_{i=1}^{i=r}\left(A_{i}(x+\kappa_{0})^{2}-\frac{q_{i}^{2}}{4A_{i}}\right)^{n_{i}}=0.$$
\section{Entropy Calculations}\label{Calcs}
\subsection{Entropy for Einstein metrics}
The value of $\nu(g)$ for an Einstein manifold $(M,g)$ with Einstein constant $1/2\tau$ is given by 
\begin{equation}\label{nuEin}
\nu(g)=\log\left(\frac{Vol(M,g)}{(4\pi\tau e)^{n/2}}\right).
\end{equation}
\subsection{Entropy for K\"ahler-Ricci solitons}
In this case we can find a reasonably explicit formula for the entropy.
\begin{proposition}
Let $(M^{n},g,f)$ be a Dancer-Wang soliton normalised to have $\tau=1$.  Let 
$$I = \int_{0}^{4}e^{-\kappa_{1}(x-2)}\prod_{i=1}^{i=r}\left(x-2-\frac{2p_{i}}{q_{i}}\right)^{n_{i}}dx$$
and let $K=2\pi\prod_{i=1}^{i=r}(-q_{i})^{n_{i}}Vol(V_{i},h_{i})$. The entropy of the soliton is given by
$$ \nu(g) = \log \left( \frac{KI}{(4\pi e)^{n/2}}\right).$$
\end{proposition}
\begin{proof}
The constant $C$ in the consistency condition can be determined by the requirement
$$(4\pi)^{-n/2}\int_{M}e^{-f}dV_{g}=1,$$
which translates into
$$K(4\pi)^{-n/2}\int_{0}^{4}e^{-\kappa_{1}(x-2)-C}\prod_{i=1}^{i=r}\left(x-2-\frac{2p_{i}}{q_{i}}\right)^{n_{i}}dx=1.$$
The constant $\kappa_{1}$ is determined by requiring
$$ \int_{0}^{4}(2-x)e^{-\kappa_{1}x}\prod_{i=1}^{i=r}\left(x-2-\frac{2p_{i}}{q_{i}}\right)^{n_{i}}dx =0.
$$
This can be rewritten as
$$ \int_{0}^{4}\kappa_{1}(x-2)e^{-\kappa_{1}(x-2)-C}\prod_{i=1}^{i=r}\left(x-2-\frac{2p_{i}}{q_{i}}\right)^{n_{i}}dx =0.
$$
Hence
$$(4\pi)^{n/2}\int_{M}fe^{-f}dV_{g} = \frac{K}{(4\pi)^{n/2}}\int_{0}^{4}Ce^{-\kappa_{1}(x+\kappa_{0})}\prod_{i=1}^{i=r}\left(x-2-\frac{2p_{i}}{q_{i}}\right)^{n_{i}}dx = C.$$
So
$$\nu(g) = C-\frac{n}{2} = \log\left(\frac{KI}{(4\pi e)^{n/2}}\right)$$
\end{proof}
\subsection{Entropy for warped product Einstein metrics}
Recall that a quasi-Einstein metric $(M^{n},g,f)$ solving (\ref{QEMeq}) and an Einstein manifold $(F^{m},h_{\mu})$, scaled so that the Einstein constant $\mu$ is given by equation (\ref{mueq}), yield a warped product Einstein manifold $M\times_{e^{-f/m}}F$ with Einstein constant $1/2\tau$. The $\nu$-entropy is given by the following:
\begin{proposition}
Let $(M^{n},g,f,m)$ be a quasi-Einstein metric described in subsection \ref{QEMans} with $m \in \mathbb{N}$ ($m\geq 2$) and let $(F^{m}, h)$ be an Einstein metric. Then the $\nu$-entropy of the warped product metric $M\times_{e^{-f/m}}F$ is given by
$$\nu(M\times_{e^{-f/m}}F) = \log\left(\frac{K\cdot Vol(F,h_{1/2\tau})}{(4\pi \tau e)^{(n+m)/2}}\left(\frac{1}{2\tau E^{\ast}}\right)^{m/2}\int_{0}^{4}(s+\kappa_{0})^{m}\prod_{i=1}^{i=r}(\beta_{i})^{n_{i}}ds\right),$$
where $K=2\pi\prod_{i=1}^{i=r}Vol(V_{i},h_{i})$ and $Vol(F,h_{1/2\tau})$ is the volume of $F$ with the Einstein metric scaled to have Einstein constant $1/2\tau$.

\end{proposition}
\begin{proof}
Formula (\ref{nuEin}) yields
$$\nu(g\oplus e^{-f/m}h) = \log \left( \frac{Vol(M\times_{e^{-f/m}}F)}{(4\pi\tau e)^{(n+m)/2}}\right).$$
The volume $Vol(M\times_{e^{-f/m}}F)$ is given by 
$$ Vol(M\times_{e^{-f/m}}F) = \int_{M}e^{-f}dV_{g}\cdot Vol(F,h_{\mu}).$$
So if we consider the volume of $F$ when the metric $h$ is scaled to have Einstein constant $1/2\tau$, we have 
$$ Vol(M\times_{e^{-f/m}}F) = \int_{M}e^{-f}dV_{g}\cdot (2\tau\mu)^{-m/2}Vol(F,h_{1/2\tau}). $$
Recalling the ansatz for the quasi-Einstein metrics
$$e^{-f} = [\kappa_{1}(s+\kappa_{0})]^{m} \textrm{ and } dV_{g}= \prod_{i=1}^{i=r}(\beta_{i})^{n_{i}}ds,$$
this yields
$$ Vol(M\times_{e^{-f/m}}F) = K\left(\frac{\kappa_{1}}{\sqrt{2\tau\mu}}\right)^{m}\int_{0}^{4}(s+\kappa_{0})^{m}\prod_{i=1}^{i=r}(\beta_{i})^{n_{i}}ds\cdot Vol(F,h_{1/2\tau}).$$
Recalling also the consistency condition $E^{\ast}=\mu/\kappa_{1}^{2}$ we can rewrite this as
$$Vol(M\times_{e^{-f/m}}F) = K\left(\frac{1}{2\tau E^{\ast}}\right)^{m/2}\int_{0}^{4}(s+\kappa_{0})^{m}\prod_{i=1}^{i=r}(\beta_{i})^{n_{i}}ds\cdot Vol(F,h_{1/2\tau}),$$
where $K=2\pi\prod_{i=1}^{i=r}Vol(V_{i},h_{i})$, this completes the proof.
\end{proof}
From this $\nu$-entropy of the warped product metric we can subtract
$$\nu(h) = \log\left( \frac{Vol(F,h_{1/2\tau})}{(4\pi \tau e)^{m/2}}\right)$$  
and this allows us to interpret the remaining quantity as an entropy for the base quasi-Einstein metric $g$. In particular as the $\nu$-entropy is additive it allows one to order the warped product metric within the product metrics.  It is important to remark that this normalised quantity is not necessarily the true $\nu$-entropy of $g$ and it would be interesting to see how it compares.
\subsection{Numerical integration} 
As the numerical results show, the differences in entropies for various metrics are often only detected at quite a number of decimal places.  We have written various routines in C++ to evaluate the integrals.  We find that the Simpson $3/8$ integration scheme with 1500 steps yields satisfactory results. The values of the various constants determining the metrics ($R, \kappa_{0},...,$etc) are given to as many decimal places as needed for the entropy to converge to an unchanging value.   
\section{Results}\label{Results}
In this section we present the numerical investigations for various manifolds.  In the \textbf{Metric} column we describe the type of metric (KRS $=$ K\"ahler-Ricci soliton, QE $=$ quasi-Einstein), in the \textbf{Data} column we give the numerical value of the particular constant that determines the metric (see section \ref{ansatz}) and obviously the $\nu$ column gives the value of the $\nu$-entropy. The value for quasi-Einstein metrics is the adaptation of the $\nu$-entropy as described above.
\subsection{The manifold $\mathbb{CP}^{1}\rightarrow \mathbb{CP}^{1}$}
The manifold $\mathbb{CP}^{1}\rightarrow \mathbb{CP}^{1}$ corresponds to data $r=1$, $n_{1}=1$, $p_{1}=2$ and $q=-1$. The Fubini-Study Einstein metric on the base $\mathbb{CP}^{1}$ is normalised so that its volume is $2\pi$. There is a $\mathbb{Z}_{2}$-invariant metric on this manifold due to Page and a K\"ahler-Ricci soliton due independently to Koiso and Cao. There are also quasi-Einstein metrics due to L\"u, Page and Pope for all $m>1$.\\
\begin{center}
\begin{tabular}{|c|c|c|} 
\hline
\textbf{Metric} & \textbf{Data}& \textbf{$\nu$}   \\
\hline
\hline
QE  $m=2$ & $\epsilon_{1}=-1, \kappa_{0} = 8.83536$ & $\log\left( 3.826565 e^{-2} \right)$ \\
\hline
QE  $m=3$ & $\epsilon_{1}=-1, \kappa_{0} = 12.76421$ & $\log\left( 3.826559 e^{-2} \right)$\\
\hline
QE  $m=4$ & $\epsilon_{1}=-1, \kappa_{0} = 16.63595$ & $\log\left( 3.826557 e^{-2} \right)$ \\
\hline
QE  $m=5$ & $\epsilon_{1}=-1, \kappa_{0} = 20.48007$ & $\log\left(3.826555 e^{-2} \right)$\\
\hline
KRS & $\kappa_{1} = 0.26381$ & $\log\left( 3.826552 e^{-2} \right)$\\
\hline
Einstein ($\mathbb{Z}_{2}$)& $R=2.10308$ & $\log\left( 3.821379 e^{-2} \right)$ \\
\hline
\end{tabular}\\
Table 1: Data and $\nu$-entropy for $r=1$, $n_{1}=1$, $q_{1}=-1$.
\end{center}
We note that the calculation of the $\nu$-entropy for the Einstein and soliton metrics agrees with the values given in \cite{CHI} and \cite{Hall11}.
\subsection{The manifold  $\mathbb{CP}^{1}\rightarrow \mathbb{CP}^{2}$ $(q=-1)$}
The  manifold  $\mathbb{CP}^{1}\rightarrow \mathbb{CP}^{2}$ is described by the data $r=1, n_{1}=2, p_{1}=3, q=-1$.  The volume of the Fubini-Study metric $h$ on $\mathbb{CP}^{2}$ is $2\pi^{2}$ when $\Ric(h)=3h$.\\
\begin{center}
\begin{tabular}{|c|c|c|} 
\hline
\textbf{Metric} & \textbf{Data}& \textbf{$\nu$}   \\
\hline
\hline
QE  $m=2$ & $\epsilon_{1}=-1, \kappa_{0} = 8.27782$  & $\log\left( 8.666758 e^{-3} \right)$ \\
\hline
QE  $m=3$ & $\epsilon_{1}=-1, \kappa_{0} = 11.52864$ & $\log\left( 8.666749 e^{-3} \right)$\\
\hline
QE  $m=4$ & $\epsilon_{1}=-1, \kappa_{0} = 14.66181$ & $\log\left( 8.666744 e^{-3} \right)$ \\
\hline
QE  $m=5$ & $\epsilon_{1}=-1, \kappa_{0} = 17.73342$ & $\log\left(8.666742 e^{-3} \right)$\\
\hline
KRS & $\kappa_{1} =  0.341008 $ & $\log\left( 8.666736 e^{-3} \right)$\\
\hline
Einstein ($\mathbb{Z}_{2}$) & $R=2.08282637 $ & $\log\left( 8.658828 e^{-3}\right)$ \\
\hline
\end{tabular}\\
Table 2: Data and $\nu$-entropy for $r=1$, $n_{1}=2$, $q_{1}=-1$.
\end{center}
\subsection{The manifold  $\mathbb{CP}^{1}\rightarrow \mathbb{CP}^{2}$ $(q=-2)$}
The  manifold  $\mathbb{CP}^{1}\rightarrow \mathbb{CP}^{2}$ is described by the data $r=1, n_{1}=2, p_{1}=3, q=-2$.  The volume of the Fubini-Study metric $h$ on $\mathbb{CP}^{2}$ is $2\pi^{2}$ when $\Ric(h)=3h$.\\
\begin{center}
{\begin{tabular}{|c|c|c|} 
\hline
\textbf{Metric} & \textbf{Data}& \textbf{$\nu$}   \\
\hline
\hline
QE  $m=2$ & $\epsilon_{1}=-1, \kappa_{0} = 2.978593$ & $\log\left(7.674619  e^{-3} \right)$ \\
\hline
QE  $m=3$ & $\epsilon_{1}=-1, \kappa_{0} = 4.435314$ & $\log\left(7.673823  e^{-3} \right)$\\
\hline
QE  $m=4$ & $\epsilon_{1}=-1, \kappa_{0} = 5.858083$ & $\log\left(7.673454e^{-3} \right)$ \\
\hline
QE  $m=5$ & $\epsilon_{1}=-1, \kappa_{0} = 7.262220$ & $\log\left(7.673249 e^{-3} \right)$\\
\hline
KRS & $\kappa_{1} =  0.735304$ & $\log\left(7.672742 e^{-3} \right)$\\
\hline
Einstein  ($\mathbb{Z}_{2}$)  & $R= 2.494993$ &  $\log\left( 7.520268 e^{-3} \right)$\\
\hline
\end{tabular}\\
Table 3: Data and $\nu$-entropy for $r=1$, $n_{1}=2$, $q_{1}=-2$}.
\end{center}
The values of $q$ in the above examples determine the Euler class of the principal $U(1)$-bundle $P_{q}$. Another way this appears in the topology of the manifolds is by giving the cohomology rings $H^{\ast}({\mathbb{CP}^{1}\rightarrow \mathbb{CP}^{2}(q);\mathbb{Z})}$ different multiplicative structures. The fact that the $q=-2$ examples have lower entropies could suggest that the Ricci flow is moving towards the $q=-1$ bundle. It would be interesting to prove that this was a general phenomenon (that lower values of $|q|$ have higher entropy for an arbitrary base). Whether this change in topology could manifest itself in terms of singularity formation, rescaling and surgery is a challenging future direction for investigation.

\subsection{The manifold $\mathbb{CP}^{1}\rightarrow \mathbb{CP}^{1}\times \mathbb{CP}^{2}$}
The manifold $\mathbb{CP}^{1}$-bundle over $\mathbb{CP}^{1}\times \mathbb{CP}^{2}$ is given by data $r=2$, $n_{1}=1$, $n_{2}=2, p_{1}=2, p_{2}=3$, ${q_{1}=-1}$, $q_{2}=-1$. The manifold admits an Einstein metric with $\mathbb{Z}_{2}$ fibre-wise symmetry, an Einstein metric without $\mathbb{Z}_{2}$ symmetry, a K\"ahler-Ricci soliton, and two families of quasi-Einstein metrics.\\
\begin{center}
\begin{tabular}{|c|c|c|} 
\hline
\textbf{Metric} & \textbf{Data}& \textbf{$\nu$}   \\
\hline
\hline
QE  $m=2$ & $\epsilon=(-1,-1), \kappa_{0} = 4.7516687 $ & $\log\left(16.60638555  e^{-4} \right)$ \\
\hline
QE  $m=3$ & $\epsilon=(-1,-1), \kappa_{0} = 6.6928512 $ & $\log\left(16.60618089  e^{-4} \right)$\\
\hline
QE  $m=4$ & $\epsilon=(-1,-1), \kappa_{0} = 8.5390242 $ & $\log\left(16.60608368 e^{-4} \right)$ \\
\hline
QE  $m=5$ & $\epsilon=(-1,-1), \kappa_{0} = 10.3319164 $ & $\log\left(16.60602849 e^{-4} \right)$\\
\hline
KRS & $\kappa_{1} = 0.60448$ & $\log\left(16.605881 e^{-4} \right)$\\
\hline
QE  $m=5$ &  $\epsilon=(1,-1), \kappa_{0}=101.0473989 $ & $\log\left( 16.60491995 e^{-4} \right)$\\
\hline
QE  $m=4$ & $\epsilon=(1,-1), \kappa_{0}=88.035526 $ & $\log\left( 16.60491993 e^{-4} \right)$ \\
\hline
QE  $m=3$ & $\epsilon=(1,-1), \kappa_{0}=74.99986 $ & $\log\left(  16.60491989 e^{-4} \right)$\\
\hline
QE  $m=2$ & $\epsilon=(1,-1), \kappa_{0}=61.925673 $ & $\log\left(16.60491982 e^{-4} \right)$ \\
\hline
Einstein & $\epsilon=(1,-1), \kappa_{0} = 35.496485$ &  $\log\left(  16.60491943e^{-4} \right)$\\
\hline
Einstein ($\mathbb{Z}_{2}$) & $R=2.1956987083$ &  $\log\left(16.53299983 e^{-4} \right)$ \\
\hline
\end{tabular}\\
Table 4: Data and $\nu$-entropy for ${r=2, n_{1}=1, n_{2}=2, p_{1}=2, p_{2}=3, q_{1}=-1, q_{2}=-1}$.
\end{center}
\emph{Acknowledgements:}
I would like to thank the anonymous referee for their careful reading of the paper and suggestions for improving it.


\begin{thebibliography}{999999}

\bibitem{BB} L. B\'erard-Bergery, Sur des Nouvelles Vari\'et\'es Riemannienes d'Einstein, Publication de l'Institute Elie Cartan, Nancy, (1982).

\bibitem{Cao} H.-D. Cao, Existence of gradient Ricci solitons, in: Elliptic and Parabolic Methods in Geometry, A.K. Peters, Wellesley, (1996), 1--16.

\bibitem{CHI} H.-D. Cao, R.S. Hamilton, T. Ilmanen, Gaussian densities and stability for some Ricci solitons, arXiv:math/0404165, (2004).

\bibitem{CV} T. Chave, G. Valent, On a class of compact and non-compact quasi-Einstein metrics and their renormalizability properties, Nuclear Phys. B, 478, no. 3, (1996) 758--778.

\bibitem{DW} A. Dancer, M. Wang, On Ricci solitons of cohomogeneity one, Ann. Global Anal. Geom., 39, (2011), 259--292.

\bibitem{Hall11} S. J. Hall, Computing Perelman's $\nu$-functional, Diff. Geom. Appl., 29, (2011), 426--432.

\bibitem{Hall13} S. J. Hall, Quasi-Einstein metrics on hypersurface families, J. Geom. Phys., 64, (2013), 83--90.

\bibitem{KK} D.-S. Kim, Y.H. Kim, Compact Einstein warped product spaces with nonpositive scalar curvature, Proc. Amer. Math. Soc., 131, (2003), 2573--2576.

\bibitem{Koi} N. Koiso, On Rotationally Symmetric Hamilton’s Equation for K\"ahler-Einstein Metrics, in: Advanced studies in Pure Mathematics, vol. 18-I, Academic Press, Tokyo, (1990), pp. 327--337.

\bibitem{LPP} H. L\"u, D. Page, C. Pope, New inhomogenous Einstein metrics on sphere bundles over Einstein-K\"ahler manifolds, Phys. Lett. B, 593, (2004), 218--226.

\bibitem{Pag} D. Page, A compact rotating gravitational instanton, Phys. Lett. B, 79, (1979), 235--238.

\bibitem{Per} G. Perelman, The entropy formula for the Ricci flow and its geometric applications, arXiv:math/0211159v1, (2002).

\bibitem{Rot} O. Rothaus, Logarithmic Sobolev inequalities and the spectrum of Schr\"odinger operators, J. Funct. Anal., no. 1, 42, (1981), 110--120.

\bibitem{Song} J. Song, G. Sz\'ekelyhidi, B. Weinkove, The K\"ahler-Ricci flow on projecive bundles, Int. Math. Res. Not., no. 2, (2013), 243--257.

\bibitem{WW} J. Wang, M. Wang, Einstein metrics on $S^{2}$-bundles, Math. Ann., 310, (1998), 497--526.
\end{thebibliography}
\end{document}